\documentclass[a4paper,10pt]{amsart}

\usepackage{amsmath,amsfonts,amssymb}
\usepackage[utf8x]{inputenc}
\usepackage[english]{babel}
\usepackage[T1]{fontenc}
\usepackage[all]{xy}
\usepackage{fullpage}

\usepackage[hypertexnames=false,
backref=page,
    pdftex,
    pdfpagemode=UseNone,
    breaklinks=true,
    extension=pdf,
    colorlinks=true,
    linkcolor=blue,
    citecolor=blue,
    urlcolor=blue,
]{hyperref}

\allowdisplaybreaks[1]

\newcommand{\referenza}{}

\newtheorem{prop}{Proposition}[section]
\newtheorem{thm}[prop]{Theorem}
\newtheorem*{thm*}{Theorem \referenza}
\newtheorem*{thmfrench*}{Th\'eor\`eme \referenza}
\newtheorem{cor}[prop]{Corollary}
\newtheorem{lemma}[prop]{Lemma}

\theoremstyle{definition}

\newtheorem{rem}[prop]{Remark}

\newcommand{\N}{\mathbb{N}}
\newcommand{\Z}{\mathbb{Z}}

\newcommand{\C}{\mathbb{C}}

\DeclareMathOperator{\imm}{im}
\DeclareMathOperator{\coker}{coker}
\DeclareMathOperator{\de}{d}

\newcommand{\del}{\partial}
\newcommand{\delbar}{\overline{\del}}

\title{On non-K\"ahler degrees of complex manifolds}

\author{Daniele Angella}

\address[Daniele Angella]{
Dipartimento di Matematica e Informatica ``Ulisse Dini''\\
Universit\`a degli Studi di Firenze\\
viale Morgagni 67/a\\
50134 Firenze, Italy
}

\email{daniele.angella@gmail.com}
\email{daniele.angella@unifi.it}

\author{Adriano Tomassini}

\address[Adriano Tomassini]{Dipartimento di Scienze Matematiche, Fisiche, ed Informatiche\\
Plesso Matematico e Informatico\\
Universit\`{a} di Parma\\
Parco Area delle Scienze 53/A, 43124\\
Parma, Italy}

\email{adriano.tomassini@unipr.it}

\author{Misha Verbitsky}
\address[Misha Verbitsky]{
Instituto Nacional de Matem\'atica Pura e Aplicada\\ Estrada Dona Castor
ina, 110\\
Jardim Bot\^anico, CEP 22460-320\\
Rio de Janeiro, RJ - Brasil\\
also:\\
Laboratory of Algebraic Geometry, \\
Faculty of Mathematics, National Research University HSE,\\
7 Vavilova Str. Moscow, Russia.
}

\email{verbit@impa.br}

\thanks{During the preparation of this paper, the first author has been granted with a research fellowship by Istituto Nazionale di Alta Matematica INdAM and supported by the Project PRIN ``Varietà reali e complesse: geometria, topologia e analisi armonica'', by the Project FIRB ``Geometria Differenziale e Teoria Geometrica delle Funzioni'', by the Project SIR 2014 AnHyC ``Analytic aspects in complex and hypercomplex geometry'' (code RBSI14DYEB), and by GNSAGA of INdAM. The second author is supported by the Project PRIN ``Varietà reali e complesse: geometria, topologia e analisi armonica'', by Project FIRB ``Geometria Differenziale Complessa e Dinamica Olomorfa'', and by GNSAGA of INdAM.
The third author is partially supported by the Russian Academic Excellence Project ’5-100’.}

\begin{document}

\begin{abstract}
 We study cohomological properties of complex manifolds. In particular, under suitable metric conditions, we extend to higher dimensions a result by A. Teleman, which provides an upper bound for the Bott-Chern cohomology in terms of Betti numbers for compact complex surfaces according to the dichotomy $b_1$ even or odd.
\end{abstract}

\maketitle

\section*{Introduction}

Let $X$ be a compact complex manifold. Consider the double complex $\left( \wedge^{\bullet,\bullet}X,\, \del,\, \delbar \right)$.
Besides de Rham and Dolbeault cohomology, consider the {\em Bott-Chern} \cite{bott-chern} and {\em Aeppli} \cite{aeppli} cohomologies, defined as:
$$ H^{\bullet,\bullet}_{BC}(X) \;:=\; \frac{\ker\del\cap\ker\delbar}{\imm\del\delbar} \qquad \text{ and } \qquad H^{\bullet,\bullet}_{A} \;:=\; \frac{\ker\del\delbar}{\imm\del+\imm\delbar} \;. $$
The Hodge theory developed by M. Schweitzer \cite{schweitzer} assures their finite-dimensionality as $\C$-vector spaces. The identity induces the natural maps
$$ \xymatrix{
 & H^{\bullet,\bullet}_{BC}(X) \ar[ld] \ar[d] \ar[rd] & \\
 H^{\bullet,\bullet}_{\del}(X) \ar[rd] & H^{\bullet}_{dR}(X;\C) \ar[d] & H^{\bullet,\bullet}_{\delbar}(X) \ar[ld] \\
 & H^{\bullet,\bullet}_{A}(X) & \\
} $$
of (bi-)graded vector spaces. One says that $X$ satisfies the {\em $\del\delbar$-Lemma} if the natural map $H^{\bullet,\bullet}_{BC}(X) \to H^{\bullet}_{dR}(X;\C)$ is injective. This turns out to be equivalent to all the above maps being isomorphisms, \cite[Lemma 5.15, 5.21, Remark 5.16]{deligne-griffiths-morgan-sullivan}. Therefore, while compact K\"ahler manifolds satisfy the $\del\delbar$-Lemma, for complex non-K\"ahler manifolds the above maps may be neither injective nor surjective. A (non-canonical) comparison between Bott-Chern and Aeppli cohomologies and de Rham cohomology is provided by the inequality {\itshape à la} Fr\"olicher in \cite[Theorem A]{angella-tomassini-3}. More precisely, for any $k\in\Z$, we have the non-negative degree
$$ \Delta^k \;:=\; \dim_\C H^k_{BC}(X) + \dim_\C H^k_{A}(X) - 2\, b_k \;\in\; \N \;, $$
where $H^k_{BC}(X):=\bigoplus_{p+q=k}H^{p,q}_{BC}(X)$, (and same for Aeppli,) and $b_k:=\dim_\C H^{k}_{dR}(X;\C)$ denotes the $k$th Betti number.
The validity of the $\del\delbar$-Lemma is characterized by $\Delta^k=0$ for any $k\in\Z$, \cite[Theorem B]{angella-tomassini-3}. Such a result is extended to generalized-complex structures, here including symplectic structures, in \cite{angella-tomassini-5, chan-suen}.
An upper-bound of the dimensions of the Bott-Chern cohomology in terms of the Hodge numbers is provided in \cite{angella-tardini-1}.

\bigskip

In this note, we study the cohomology of compact complex manifolds.

In general, the degrees $\Delta^k$ measure the failure of $\del\delbar$-Lemma, that is, non-cohomologically-K\"ahlerness. In fact, they measure non-K\"ahlerness for compact complex surfaces. This is because of the topological characterization of K\"ahlerness in terms of the parity of the first Betti number, \cite{kodaira-1, miyaoka, siu}, see also \cite[Corollaire 5.7]{lamari}, and \cite[Theorem 11]{buchdahl}.

In \cite{teleman}, (see also \cite{lubke-teleman},) A. Teleman proves that, for compact complex surfaces, $\Delta^1$ is always zero and $\Delta^2\in\{0,2\}$; and, in \cite{angella-dloussky-tomassini}, the Bott-Chern and Aeppli cohomologies for compact complex surfaces diffeomorphic to solvmanifolds and for class $\text{VII}$ surfaces are computed, here including Inoue and Kodaira surfaces, \cite[Theorem 4.1]{angella-dloussky-tomassini}, and class $\text{VII}$ surfaces, \cite[Theorem 2.2]{angella-dloussky-tomassini}.

\begin{thm}[{\cite[Lemma 2.3]{teleman}}]
 Let $X$ be a compact complex surface. Then:
 \begin{itemize}
  \item $b_1$ is even if and only if $\Delta^1=\Delta^2=0$;
  \item $b_1$ is odd if and only if $\Delta^1=0$ and $\Delta^2=2$.
 \end{itemize}
\end{thm}

We notice that this result provides an answer to a question by A. Fujiki, asking whether $\Delta^2$ may change under deformations of the complex structure.

\begin{cor}
 For compact complex surfaces, $\Delta^1$ and $\Delta^2$ are topological invariants.
\end{cor}

Note that this is no more true in higher-dimension: see the examples on the Iwasawa manifold in \cite{angella-1}, and the examples on the Nakamura manifold in \cite{angella-kasuya-1, angella-kasuya-2}.

\medskip

In higher dimension, we get a result concerning the first degree $\Delta^1$ under additional assumptions concerning the existence of special Hermitian metrics. 

\renewcommand{\referenza}{\ref{thm:delta-1-higher}}
\begin{thm*}
 Let $X$ be a compact complex manifold of complex dimension $n$ endowed with a Hermitian metric $g$. Suppose that its associated $(1,1)$-form $\omega$ satisfies either the condition that $\de\omega^{n-2}\in\imm\de\de^c$, or the condition that $\omega^{n-2}$ is the $(n-2,n-2)$-component of a $\de$-exact $(2n-4)$-form. (In particular, such an $\omega$ is astheno-K\"ahler in the sense of Jost and Yau \cite{jost-yau}.)
 Then $\Delta^1=0$.
\end{thm*}

Note that the assumption is trivially satisfied for compact complex surfaces.
On the other side, as in \cite{angella-tomassini-1, enrietti-fino-vezzoni}, we show that $6$-dimensional nilmanifolds endowed with left-invariant complex structures never admit Hermitian metrics as above, see Proposition \ref{prop:nilmfd}.
Notwithstanding, there are examples of $6$-dimensional manifolds with $\Delta^1=0$, see \cite{angella-franzini-rossi, latorre-ugarte-villacampa}.

\bigskip

\noindent{\sl Acknowledgments.} The authors are grateful to Ionut Chiose, Georges Dloussky, Akira Fujiki, and Andrei Teleman for valuable comments, useful conversations, and for their interest in the subject.
In particular, Ionut Chiose pointed out to us some improvements and mistakes in a previous draft.
Thanks also to the anonymous Referees for their comments and suggestions.

\section{Non-K\"ahlerness \texorpdfstring{$1$}{1}st degree for higher-dimensional complex manifolds}

Let $X$ be a compact complex manifold of complex dimension $n$ endowed with a Hermitian metric $g$. Denote by $\omega$ its associated $(1,1)$-form. Recall that
$$ \mathcal{D} \colon \mathcal{C}^\infty(X;\C) \to \wedge^{2n}X \otimes \C \;, \qquad \mathcal{D}(f) \;:=\; \de\de^cf\wedge\omega^{n-1} $$
is an elliptic differential operator with index zero and $1$-dimensional kernel, \cite{gauduchon-cras-1975}.

Define the Hermitian degree
$$ \deg \colon H^{1,1}_{BC}(X) \to \frac{\wedge^{2n}X\otimes\C}{\imm\mathcal{D}} \;\simeq\; \C \;, \qquad \deg([\alpha]) \;:=\; \frac{\alpha\wedge \omega^{n-1}}{\omega^n} \;. $$

\begin{lemma}\label{lemma:deg-inj-on-exact}
 Let $X$ be a compact complex manifold of complex dimension $n$ endowed with a Hermitian metric $g$. Suppose that its associated $(1,1)$-form $\omega$ satisfies the following condition:
 \begin{description}
  \item[(a)] $\omega^{n-2}$ is the $(n-2,n-2)$-component of a $\de$-closed $(2n-4)$-form.
 \end{description}
If $[\de\alpha]\in H^{1,1}_{BC}(X)$ is such that $\deg([\de\alpha])=0$, then $[\de\alpha]=0$.
\end{lemma}

\begin{proof}
 By the hypothesis: 
 take $f\in\mathcal{C}^\infty(X;\C)$ such that
 $$ \de \alpha \wedge \omega^{n-1} \;=\; \de\de^c f \wedge \omega^{n-1} \;. $$
  
 Set $\alpha':=\alpha-\de^c f$. Note that $[\de\alpha]=[\de\alpha']$ in $H^{1,1}_{BC}(X)$, and that $\de\alpha'$ is primitive. Hence use Weil identity to write $*\de\alpha'=\frac{1}{(n-2)!}J\de\alpha'\wedge\omega^{n-2}=\frac{1}{(n-2)!}\de\alpha'\wedge\omega^{n-2}$.
  
 Under the assumption {\itshape (a)}, take $\beta\in\wedge^{2n-4}$ such that: (here $\pi_{\wedge^{n-2,n-2}X}$ denotes the natural projections onto $\wedge^{n-2,n-2}X$)
 $$ \omega^{n-2} \;=\; \pi_{\wedge^{n-2,n-2}X}\beta \qquad \text{ with } \qquad \de\beta \;=\;0 \;. $$
 We have:
 \begin{eqnarray*}
  \| \de\alpha' \|^2 &=& \int_X \de\alpha' \wedge \overline{*(\de\alpha')} \;=\; -\frac{1}{(n-2)!} \int_X \de\alpha' \wedge \de\bar\alpha' \wedge \omega^{n-2} \\[5pt]
  &=& - \frac{1}{(n-2)!} \int_X \de\alpha' \wedge \de\bar\alpha' \wedge \pi_{\wedge^{n-2,n-2}X}\beta \\[5pt]
  &=& - \frac{1}{(n-2)!} \int_X \de\alpha' \wedge \de\bar\alpha' \wedge \beta \\[5pt]
  &=& {\frac{1}{(n-2)!}} \int_X \alpha' \wedge \de\bar\alpha' \wedge \de\beta \;=\; 0 \;,
 \end{eqnarray*}
 thus giving $\de\alpha'=0$.
\end{proof}

\begin{rem}
 Note that the condition {\itshape (a)} in Lemma \ref{lemma:deg-inj-on-exact} yields that $\de\de^c\omega^{n-2}=0$. That is, $\omega$ is {\em astheno-K\"ahler} in the sense of J. Jost and S.-T. Yau, \cite{jost-yau}. Note that the condition is trivially satisfied in case $2n=4$.
 In a sense, condition {\itshape (a)} is the $(n-2)$-degree counterpart of Hermitian-symplectic condition in the sense of \cite{streets-tian}. Note that a Hermitian metric satisfying $\de\omega^{n-2}=0$ is actually K\"ahler, \cite{gray-hervella}.
\end{rem}

\begin{rem}[Ionut Chiose]
Condition {\itshape (a)} is satisfied {\itshape e.g.} when
 \begin{description}
  \item[(b)] $\de\omega^{n-2}\in\imm\de\de^c$.
 \end{description}
\end{rem}

\begin{proof}
 Condition {\itshape (b)} assures that $\partial\omega^{n-2}=\partial\overline\partial\eta$ where $\eta\in\wedge^{n-2,n-3}X$. Then $\beta:=\partial\eta+\omega^{n-2}+\overline\partial\eta$ is a $(2n-4)$-form such that $\de\beta=0$ and whose $(n-2,n-2)$-component is $\omega^{n-2}$.
\end{proof}

The following generalizes the result in \cite{angella-tomassini-1}, see also \cite{enrietti-fino-vezzoni}, in proving that conditions {\itshape (a)} and {\itshape (b)} are not satisfied for $6$-dimensional nilmanifolds with invariant structures.

\begin{prop}\label{prop:nilmfd}
 On a $6$-dimensional non-torus nilmanifold endowed with a left-invariant complex structure, there is no left-invariant metric with the property {\itshape (a)} in Lemma \ref{lemma:deg-inj-on-exact}.
\end{prop}

\begin{proof}
 In dimension $6$, condition {\itshape (a)} is equivalent to have a Hermitian-symplectic structure. This is proven to be impossible on $6$-dimensional nilmanifold in \cite[Theorem 3.3]{angella-tomassini-1}, except the torus; see also \cite[Theorem 1.3]{enrietti-fino-vezzoni}.
\end{proof}

\begin{rem}
Note that there are examples of $6$-dimensional nilmanifolds such that $\Delta^1=0$, see \cite{angella-franzini-rossi, latorre-ugarte-villacampa}.
\end{rem}

\begin{lemma}\label{_prod_4-form_Lemma_}
Suppose that the Hermitian $(1,1)$-form $\omega$ on a complex manifold $X$ satisfies either condition {\itshape (b)} in Lemma \ref{lemma:deg-inj-on-exact}, or condition
\begin{description}
  \item[(a')] $\omega^{n-2}$ is the $(n-2,n-2)$-component of a $\de$-exact $(2n-4)$-form.
\end{description}
Then a product of a $\de^c$-closed and $\de$-exact form $\de\nu$ with $\omega^{n-2}$ is exact.
\end{lemma}

\begin{proof}
Indeed, in case the assumption {\itshape (b)} is true, then
$\de\omega^{n-2} =\de\de^c\mu$, giving
\[ \int_X \de\nu \wedge \omega^{n-2}= -\int_X\nu\wedge \de\omega^{n-2}=
-\int_X\nu\wedge \de\de^c \mu=\int_X \de\de^c \nu \wedge \mu=0,
\]
which proves the statement.

In case the assumption {\itshape (a')} is true, let $\omega^{n-2}=(\de\kappa)^{(n-2,n-2)}$,
and 
\[ \int_X \de\nu \wedge \omega^{n-2}=  \int_X (\de\nu)^{(2,2)} \wedge \omega^{n-2}=
\int_X (\de\nu)^{(2,2)}\wedge \de\kappa
\]
because $\de\nu$ is both $\de$-closed and $\de^c$-closed, whence $(\de\nu)^{(2,2)}$ is $\de$-closed, and $\de\kappa$ is exact. 
\end{proof}

\begin{thm}\label{thm:delta-1-higher}
 Let $X$ be a compact complex manifold endowed with a Hermitian metric $g$.
 Suppose that the associated $(1,1)$-form $\omega$ satisfies either condition {\itshape (b)} in Lemma \ref{lemma:deg-inj-on-exact} or condition {\itshape (a')} in Lemma \ref{_prod_4-form_Lemma_}.
 Then $\Delta^1=0$.
\end{thm}

\begin{proof}
 We claim that, under the hypotheses, the sequence
 \begin{equation}\label{_exact_Aeppli_H^1_Equation_} 
 0 \to H^1_{dR}(X;\C) \to H^1_{A}(X) \stackrel{\deg\circ\de}{\to} \C
 \end{equation}
 is exact.

 Indeed, take $[\alpha]_{dR}\in H^1_{dR}(X;\C)$ yielding a zero class in Aeppli cohomology, that is, $\alpha=\del f + \delbar g$ for $f,g\in\mathcal{C}^\infty(X;\C)$. Since $\de\alpha=0$, then $\del\delbar (f-g)=0$. By the maximum principle, we get $f-g$ constant. Then $\alpha=\de f$ yields the zero class in de Rham cohomology.
We have that \eqref{_exact_Aeppli_H^1_Equation_} is a complex, and it remains to show that all classes in $H^1_A(X)$
of degree 0 come from $H^1_{dR}(X)$.

Let $[\alpha]_{A} \in  H^1_{A}(X)$ satisfy $\deg[\de\alpha]_{BC}=0$.
Choose a representative $\alpha\in \wedge^1 X\otimes \C$.
Let $\rho=\de\alpha$ be decomposed as $\rho= \rho^{(1,1)}+\rho^{(2,0)}+\rho^{(0,2)}$.
Without restricting the generality we may assume that $\rho$ is real, giving $\rho^{(2,0)}= \overline{\rho^{(2,0)}}$.
Since $\alpha$ is $\de\de^c$-closed, $\de\alpha$ is both $\de$ and $\de^c$-closed, hence the forms $\rho^{(1,1)}$, $\rho^{(2,0)}$ and $\rho^{(0,2)}$ are closed.
Since $\rho$ is $\de^c$-closed and $\de$-exact, the form $\rho\wedge \rho^{(0,2)}$ is also $\de^c$-closed and $\de$-exact.
However, a product of a $\de^c$-closed and $\de$-exact $4$-form $\de\nu$ with $\omega^{n-2}$ is exact by Lemma \ref{_prod_4-form_Lemma_} above.
This gives
\[ 0 = 
   \int_X \rho\wedge \rho^{(0,2)}\wedge \omega^{n-2}
=\int_X \rho^{(0,2)}\wedge \rho^{(2,0)}\wedge \omega^{n-2}= \int_X |\rho^{(0,2)}|^2 \omega^n.
\]
This implies that $\rho=\de\alpha$ is of type $(1,1)$.
By Lemma \ref{lemma:deg-inj-on-exact}, we can replace the representative $\alpha$ by $\alpha-\de^c f$ in Aeppli cohomology in such a way that $\de(\alpha-d^c f)=0$, thus proving the claim.

 It follows that
 $$ \dim_\C H^{1}_{A}(X) - b_1 \;=\; 1 - \dim_\C\coker(\deg\circ\de) \;\in\; \{0,1\} $$
 according to the parity of $b_1$. Indeed, $\dim_\C H^1_{A}(X)$ is always even.

 Note that the natural map $H^1_{BC}(X) \to H^1_{dR}(X;\C)$ is always injective. Then
 \begin{eqnarray*} 
  0 \;\leq\; \Delta^1 &=& \dim_\C H^1_{A}(X) + \dim_\C H^1_{BC}(X) - 2\, b_1 \\[5pt]
  &\leq& \left( 1 - \dim_\C\coker(\deg\circ\de) \right) \;\leq\; 1 \;.
 \end{eqnarray*}
 Since $\Delta^1$ has to be even, this yields $\Delta^1 = 0$.
\end{proof}

\begin{rem}
 The same statement can be proven under the hypotheses that the Hermitian $(1,1)$-form $\omega$ satisfies condition {\itshape (a)} in Lemma \ref{lemma:deg-inj-on-exact} and that $H^{2,0}_{BC}(X)=0$.
In this case, the exactness at $H^1_{A}(X)$ follows directly by Lemma \ref{lemma:deg-inj-on-exact}, noting that $[\de\alpha]_{BC}\in H^{2}_{BC}(X)=H^{1,1}_{BC}(X)$.
 The condition $H^{2,0}_{BC}(X)=0$ may hold, for example, for a hypothetical complex structure on the six-sphere with $H^{2,0}_{\overline\partial}(X)=0$, see {\itshape e.g.} \cite{ugarte, mchugh-1}. In low dimension, the condition is satisfied by the Bombieri-Inoue and Inoue surfaces, by the secondary Kodaira surfaces, and by the Calabi-Eckmann structure on $\mathbb{S}^1\times\mathbb{S}^3$, see {\itshape e.g.} \cite{angella-dloussky-tomassini}.
\end{rem}

\end{document}